\documentclass{amsart}

\pdfoutput=1

\usepackage{amssymb,amsthm}
\usepackage{enumitem}
\usepackage{graphicx}
\usepackage{mathptmx}
\usepackage{natbib}
\usepackage{verbatim}

\newtheorem{definition}{Definition}

\newtheorem{lemma}{Lemma}
\newtheorem{proposition}{Proposition}
\newtheorem{theorem}{Theorem}

\begin{document}

\title{Exchangeability and the Law of Maturity}

\author{Fernando V. Bonassi}
\address{Google Inc \\
	 Mountain View, CA 94043, USA}
\email{bonassi@gmail.com}
\thanks{Partially supported by CNPq and FAPESP (2003/10105-2). The authors thank Dani Gamerman, Jay Kadane, Carlos Pereira, Teddy Seidenfeld, Julio Stern and Robert Winkler for insightful remarks.}

\author{Rafael B. Stern}
\address{Department of Statistics \\
	 Carnegie Mellon University \\ 
	 Pittsburgh, PA 15217, USA}
\email{rbstern@gmail.com}

\author{Sergio Wechsler}
\address{Instituto de Matem\'{a}tica e Estat\'{i}stica \\
	 Universidade de S\~{a}o Paulo \\
	 S\~{a}o Paulo, SP, Brazil}
\email{sw@ime.usp.br}

\author{Claudia M. Peixoto}
\address{Instituto de Matem\'{a}tica e Estat\'{i}stica \\
	 Universidade de S\~{a}o Paulo\\
	 S\~{a}o Paulo, SP, Brazil}
\email{claudia@ime.usp.br}

\begin{abstract}
The \textit{law of maturity} is the belief that less-observed events are becoming mature and, therefore, more likely to occur in the future. Previous studies have shown that the assumption of infinite exchangeability contradicts the \textit{law of maturity}. In particular, it has been shown that infinite exchangeability contradicts probabilistic descriptions of the \textit{law of maturity} such as the \textit{gambler's belief} and the \textit{belief in maturity}. We show that the weaker assumption of finite exchangeability is compatible with both the \textit{gambler's belief} and \textit{belief in maturity}. We provide sufficient conditions under which these beliefs hold under finite exchangeability. These conditions are illustrated with commonly used parametric models.
\end{abstract}

\maketitle

\section{Introduction}

The \textit{law of maturity (of chances)} is the belief that less-observed events are mature and, therefore, more likely to occur in the future. The law of maturity has been observed in fields as varied as behavioral finance \citep{rabin2010}, cognitive psychology \citep{militana2010} and decision-making \citep{oppenheimer2009}.  
Simple examples of the law of maturity are:

\begin{itemize} 
  \item The belief that successive increments in a stock price will eventually cause it to decrease.
  \item The belief that the chance of obtaining tails in a coin flipping process increases as the proportion of observed heads increases.
\end{itemize}

This paper focuses on finding conditions such that the \textit{law of maturity} holds under a model of symmetry between observations. A commonly used statistical model of symmetry is that of infinite exchangeability \citep{bernardo94}. Hence, a possible initial consideration is whether infinite exchangeability is compatible with the \textit{law of maturity}.

This question is discussed in \citet{neill2005}[hereafter, OP] and in \citet{rodrigues1993}[hereafter, RW]. Each paper provides a different probabilistic description of the \textit{law of maturity} and shows that, if a model assumes this description, then it is inconsistent with infinite exchangeability. In particular, both papers show that infinite exchangeability implies the \textit{law of reverse maturity} - less observed events are less likely to occur in the future.

In the following, the aforementioned results are presented in greater detail. This presentation prepares some of the definitions that are used in this paper. In order to do so, consider that the sequence of discrete random variables $\texttt{X} = (X_{1},X_{2},\ldots)$ can be observed sequentially and let $\texttt{X}_{n} = (X_{1},X_{2},\ldots,X_{n})$.

Let $X_{i}$ take value on $\{1,2,\ldots,k\}$. Consider the following probabilistic descriptions of, respectively, the \textit{law of maturity} and the \textit{law of reverse maturity}: 

\vspace{2mm}

\begin{definition}[gambler's belief]
  \label{def:gambler}
  The predictive probability $P_{X_{n+1}|\texttt{X}_{n}}$ of the next possible outcome is ordered inversely to the counts in $\texttt{X}_{n}$.
\end{definition}

\begin{definition}[reverse gambler's belief]
  \label{def:reverse_gambler}
  The predictive probability $P_{X_{n+1}|\texttt{X}_{n}}$ of the next possible outcome is ordered according to the counts in $\texttt{X}_{n}$.
\end{definition}

\vspace{2mm}

If $\texttt{X}$ is infinitely exchangeable and the prior distribution for the parameter induced by de Finetti's theorem \citep{finetti1931} is exchangeable, then the \textit{reverse gambler's belief} holds (OP). Hence, the assumption of infinite exchangeability contradicts the \textit{law of maturity} in this scenario.

In another scenario, let $X_{i}$ take value on $\{0,1\}$. Consider the following probabilistic descriptions of, respectively, the \textit{law of maturity} and the \textit{law of reverse maturity}:

\begin{definition}[belief in maturity]
  \label{def:maturity}
  The longer the streak of observed $0$'s, the higher the probability that the next observation is a $1$. That is, one believes that $P(X_{n+1}=1|x_{n}=0,\ldots,x_{1}=0)$, increases with $n$.
\end{definition}

\begin{definition}[belief in reverse maturity]
  \label{def:reverse_maturity}
  The predictive probability for the end of as streak, $P(X_{n+1}=1|x_{n}=0,\ldots,x_{1}=0)$, decreases with $n$.
\end{definition}

If $\texttt{X}$ is infinitely exchangeable, then the \textit{belief in reverse maturity} holds (RW). Thus, infinite exchangeability contradicts the \textit{law of maturity} also in this scenario.

The aforementioned results show that the \textit{law of maturity} contradicts infinite exchangeability, an assumption of symmetry commonly used in Bayesian studies \citep{lindley1976}. This contradiction motivates the search for weaker symmetry assumptions that support the \textit{law of maturity}. 

This paper relates the \textit{law of maturity} and the assumption of finite exchangebility. Section \ref{model} describes the assumption of finite exchangeability. Section \ref{results} considers this assumption and presents sufficient conditions for, respectively, the \textit{law of maturity} and \textit{law of reverse maturity} to hold. These conditions are exemplified with commonly used parametric models.  
  
\section{Finitely exchangeable models}
\label{model}

The assumption of infinite exchangeability contradicts the \textit{law of maturity} (RW, OP). Is this law consistent with weaker assumptions of symmetry? One such assumption is finite exchangeability \citep{mendel1994,pilar2009}:

\begin{definition}[model of finite exchangeability]
  \label{def:finite_exchangeability}
  Let $\{1,\ldots,N\}$ be indexes of the members of a given population. Let $\texttt{X}_{N} = (X_{1}, X_{2},\ldots, X_{N})$ be a vector of random variables such that $X_{i} \in \{0,1\}$. Let $\gamma(\texttt{X}_{N}) = \sum_{i=1}^{N}{X_{i}}$, that is, the count of $1$'s in $\texttt{X}_{N}$. A sample is a vector of length $n$ that is observed. Denote the observed sample by the vector $\texttt{x}_{n} = (x_{1},x_{2},\ldots,x_{n})$. $\texttt{X}_{N}$ is exchangeable if, for every permutation of $\{1,\ldots,N\}$, $\psi$, and for every $\texttt{x}_{N} \in \{0,1\}^{N}$, 
  \[ P(X_{1} = x_{1}, \ldots, X_{N} = x_{N}) = P(X_{1} = x_{\psi(1)}, \ldots, X_{N} = x_{\psi(N)}) \]  
\end{definition}

This model embraces two opposite tendencies. First, $\sum_{i=1}^{n}{X_{i}}|\gamma=\gamma_{0}$ is distributed according to a Hypergeometric($N,\gamma_0,n$), the distribution obtained by sampling without replacement. If a sample is drawn without replacement, any pair of trials has negative correlation. This fact is favorable to the \textit{law of maturity}. Second, by the use of a Bayesian approach, it is possible to learn about $\gamma$. This fact favors the \textit{law of reverse maturity} (RW, OP). These diverging tendencies suggest that the assumption of finite exchangeability is consistent with both laws.

The following section considers finitely exchangeable models and presents sufficient conditions for, respectively, the \textit{law of maturity} and the \textit{law of reverse maturity} to hold. Observe that the assumption of finite exchangeability and a prior on $\gamma$ completely specify the distribution of $\texttt{X}_{N}$. Hence, the conditions in the next section are expressed in terms of the prior distribution for $\gamma$. 

\section{Results}
\label{results}

\subsection{Indifferent Belief}
\label{sec:indifferent}

This subsection studies the borderline between the \textit{law of maturity} and the \textit{law of reverse maturity}. We choose the following probabilistic description of this borderline:

\begin{definition}[indifferent belief]
  \label{def:indifferent}
  There exists $\pi \in [0,1]$ such that, for every $\texttt{x}_{n} \in \{0,1\}^{n}$ and $n < N$, $p(X_{n+1}=1|\texttt{x}_{n}) = \pi$. The predictive probability does not depend on the observations.
\end{definition}

The following results characterize \textit{indifferent belief} according to the prior distribution of $\gamma$:

\begin{proposition}
  \label{indifferent_belief_equiv_1}
  There exists $\pi \in [0,1]$ such that $\gamma \sim \text{Binomial}(N, \pi)$ iff the coordinates of $\texttt{X}_{N}$ are jointly independent.
\end{proposition}

\begin{proposition}
 \label{indifferent_belief_equiv_2}
 Indifferent belief occurs iff there exists $\pi \in [0,1]$ such that $\gamma \sim \text{Binomial}(N,\pi)$. 
\end{proposition}

It follows from Proposition \ref{indifferent_belief_equiv_1} that, when $\gamma$ is distributed according to a Binomial($N$, $\pi$), the distribution of $\texttt{X}_{N}$ is the same as the one obtained by sampling with replacement from a known population. Proposition \ref{indifferent_belief_equiv_2} establishes a one to one correspondence between this type of sampling and indifferent belief. The next subsections use Proposition \ref{indifferent_belief_equiv_2} as guiding intuition to present sufficient conditions for the \textit{gambler's belief} and \textit{belief in maturity} to hold. 

\subsection{(Reverse) Gambler's Belief}

The model of finite exchangeability supports both the \textit{gambler's belief} and the \textit{reverse gambler's belief}. If $\gamma$ follows a degenerate distribution on $N/2$, then $\texttt{X}_{N}$ follows the same distribution as sampling without replacement from a known population. In this case, the \textit{gambler's belief} holds. On the other hand, if $\gamma$ follows an uniform distribution on $\{0,1,\ldots,N\}$, then $\texttt{X}_{N}$ can be extended to a model of infinite exchangeability. Hence, it follows from OP that the \textit{reverse gambler's belief} holds. Finally, as shown in subsection \ref{sec:indifferent}, the Binomial$(N,1/2)$ distribution is between the two previous cases and leads to \textit{indifferent belief}. 

\begin{figure}
  \centering
  \includegraphics[height=40mm,width=\columnwidth]{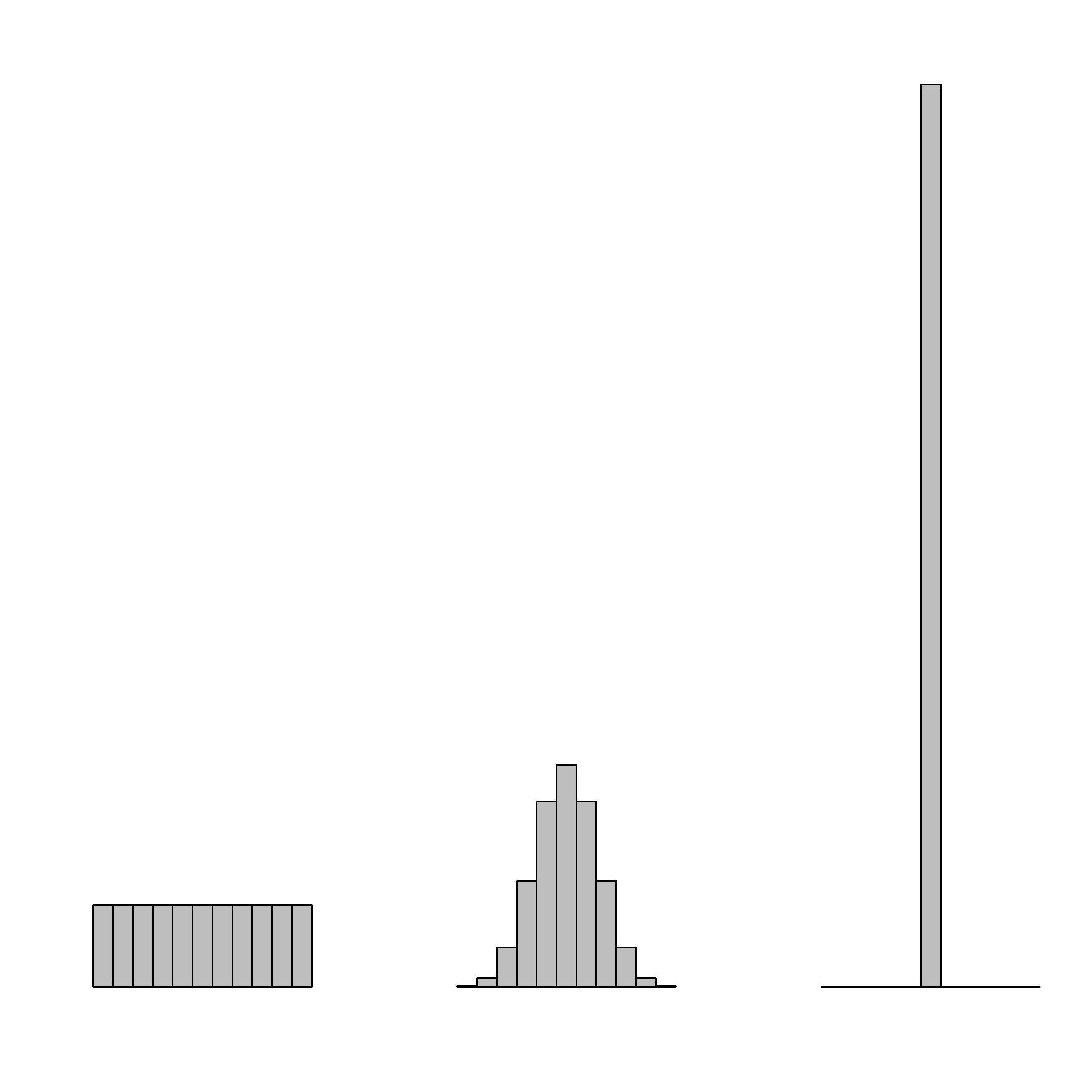} 
  \caption{Priors for $\gamma$ presented using the same vertical axis: (left) uniform distribution, (middle) binomial distribution and (right) degenerate distribution.} 
  \label{fig:intuition}
\end{figure}

Figure \ref{fig:intuition} presents a progressive pattern from the \textit{gambler's belief} (uniform distribution) to the \textit{reverse gambler's belief} (degenerate distribution). Definition \ref{def:tighter} presents conditions that generalize the pattern in Figure \ref{fig:intuition}.

\begin{definition} Let $\gamma \in \{0,1,\ldots,N\}$ have a distribution that is symmetric with respect to $N/2$. The random variable $\gamma$ is
  \label{def:tighter}
  \begin{enumerate}
    \item \textit{tighter than the Binomial$(N, 1/2)$}, if $\frac{P(\gamma=i)}{P(\gamma=i-1)} > \frac{N-i+1}{i}$ for $1 \leq i \leq \left\lfloor N/2 \right\rfloor$.
    \item \textit{looser than the Binomial$(N, 1/2)$}, if $\frac{P(\gamma=i)}{P(\gamma=i-1)} < \frac{N-i+1}{i}$ for $1 \leq i \leq \left\lfloor N/2 \right\rfloor$.
  \end{enumerate}
\end{definition}

Figure \ref{fig:tighter-to-looser} illustrates the conditions in Definition \ref{def:tighter}. The comparison of Figures \ref{fig:intuition} and \ref{fig:tighter-to-looser} shows that: (1) distributions that are looser than the Binomial are between the uniform and the Binomial, (2) distributions that are tighter than the Binomial are between the Binomial and the degenerate. Hence, one might expect that distributions that are tighter (looser) than the Binomial lead to the \textit{(reverse) gambler's belief}. This expectation is confirmed by the following theorem:

\begin{figure}
  \centering
  \includegraphics[height=40mm,width=1\columnwidth]{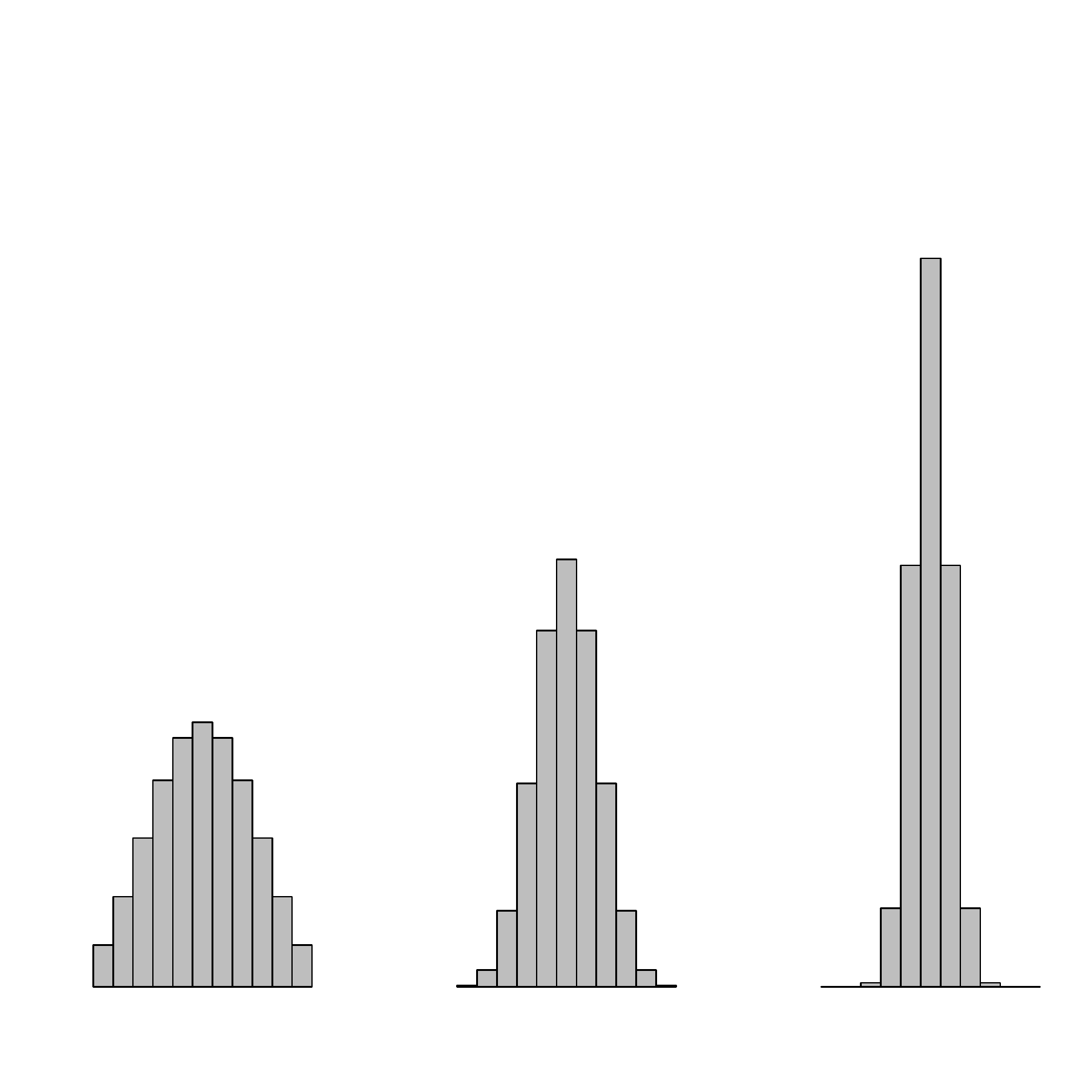} 
  \caption{Illustration of priors for $\gamma$ using the same vertical axis: (left) distribution looser than Binomial, (middle) Binomial distribution and (right) distribution tighter than Binomial.}
  \label{fig:tighter-to-looser}
\end{figure}

\begin{theorem}[gambler's belief]
  \label{theorem:gambler}
  Assume that the distribution of $\gamma$ is symmetric. The \textit{gambler's belief} holds iff $\gamma$ is tighter than the Binomial($N$, 1/2). The \textit{reverse gambler's belief} holds iff $\gamma$ is looser than the Binomial($N$, 1/2).
\end{theorem}

The assumption in Theorem \ref{theorem:gambler} that $\gamma$ is symmetric is similar to the assumption of exchangeability of de Finetti's representation parameter \citep{finetti1931} in OP. If $\texttt{X}_{N}$ is part of an infinitely exchangeable sequence $\texttt{X}$, then the exchangeability of de Finetti's representation parameter implies that the distribution of $\gamma$ is symmetric. This relation is further considered in Proposition \ref{proposition:extendibility-gambler} that shows that Theorem \ref{theorem:gambler} extends OP in the case of $0-1$ variables. 

\begin{proposition}
  \label{proposition:extendibility-gambler}
  If $\texttt{X}_{N}$ can be extended to an infinitely exchangeable sequence and de Finetti's representation parameter is exchangeable and not degenerate, then $\gamma$ is looser than the Binomial$(N,1/2)$. In particular, if $\gamma$ is distributed according to a Beta-Binomial$(N,\alpha,\alpha)$, then $\gamma$ is looser than the Binomial$(N,1/2)$.
\end{proposition}

The Beta-Binomial is widely used in Bayesian analysis \citep{lee99} and has useful properties. For example, this prior leads to an analytical solution for the posterior distribution. Nevertheless, Theorem \ref{theorem:gambler} shows that the Beta-Binomial prior is inconsistent with the \textit{gambler's belief}.

The CMP-Binomial \citep{shmueli2005} has been proposed in order to accomodate under-dispersion and over-dispersion with respect to the Binomial distribution. The parameter $\gamma$ follows the CMP-Binomial$(N,p,\nu)$ whenever $P(\gamma=i) \propto {N \choose i}^{\nu}p^{i}(1-p)^{N-i}$. It follows from definition \ref{def:tighter} that

\begin{proposition}
  \label{proposition:CMP-Binomial-1}
  Let $\gamma \sim \text{CMP-Binomial}(N,1/2,\nu)$. If $\nu > 1$, then $\gamma$ is tighter than the Binomial$(N,1/2)$. If $\nu < 1$, then $\gamma$ is looser than the Binomial$(N,1/2)$.
\end{proposition}

Hence, it follows from Theorem \ref{theorem:gambler} and Proposition \ref{proposition:CMP-Binomial-1} that, according to the choice of $\nu$, the CMP-Binomial class can be used to model the \textit{gambler's belief} or the \textit{reverse gambler's belief}.

The next subsection considers the other probabilistic description of the \textit{law of maturity} available in the literature: \textit{belief in maturity}. 

\subsection{(Reverse) Belief in Maturity}

Figure \ref{fig:intuition} is used as a starting point for exploring \textit{belief in maturity}. Recall that the degenerate distribution is an extreme case in which the \textit{belief in maturity} holds. Assume that $\texttt{X}_{N}$ is part of a larger exchangeable population $\texttt{X}_{N+M}$ such that $\gamma(\texttt{X}_{N+M})$ has a degenerate distribution. In this case, \textit{belief in maturity} holds for $\texttt{X}_{N+M}$ and, therefore, it also holds for $\texttt{X}_{N}$. Observe that $\gamma(\texttt{X}_{N})|\gamma(\texttt{X}_{N+M})=a \sim \text{Hypergeometric}(N+M,a,N)$. Hence, if $\gamma(\texttt{X}_{N})$ follows the Hypergeometric distribution, then \textit{belief in maturity} holds.

Given the previous result, we look among the properties of the Hypergeometric distribution for conditions that are sufficient for the \textit{belief in maturity}. Let $\frac{P(\gamma=i)/P(\gamma=i-1)}{P(\gamma=i+1)/P(\gamma=i)}$ be the \textit{tightness ratio} of $\gamma$. The \textit{tightness ratio} of the $Binomial(n,p)$ is $\frac{(i+1)(N-i+1)}{i(N-i)}$ and doesn't depend on $p$. Figure \ref{fig:intuition2} illustrates the \textit{tightness ratio} of the Binomial and the Hypergeometric. For every $N,M,p$ and $N \leq a \leq M$, it follows that the \textit{tightness ratio} of the Hypergeometric$(N+M,a,N)$ is larger than that of the Binomial$(N,p)$. Definition \ref{def:2nd-order-tighter} builds on this property of the Hypergeometric distribution.

\begin{figure}
  \centering
  \includegraphics[height=40mm,width=\columnwidth]{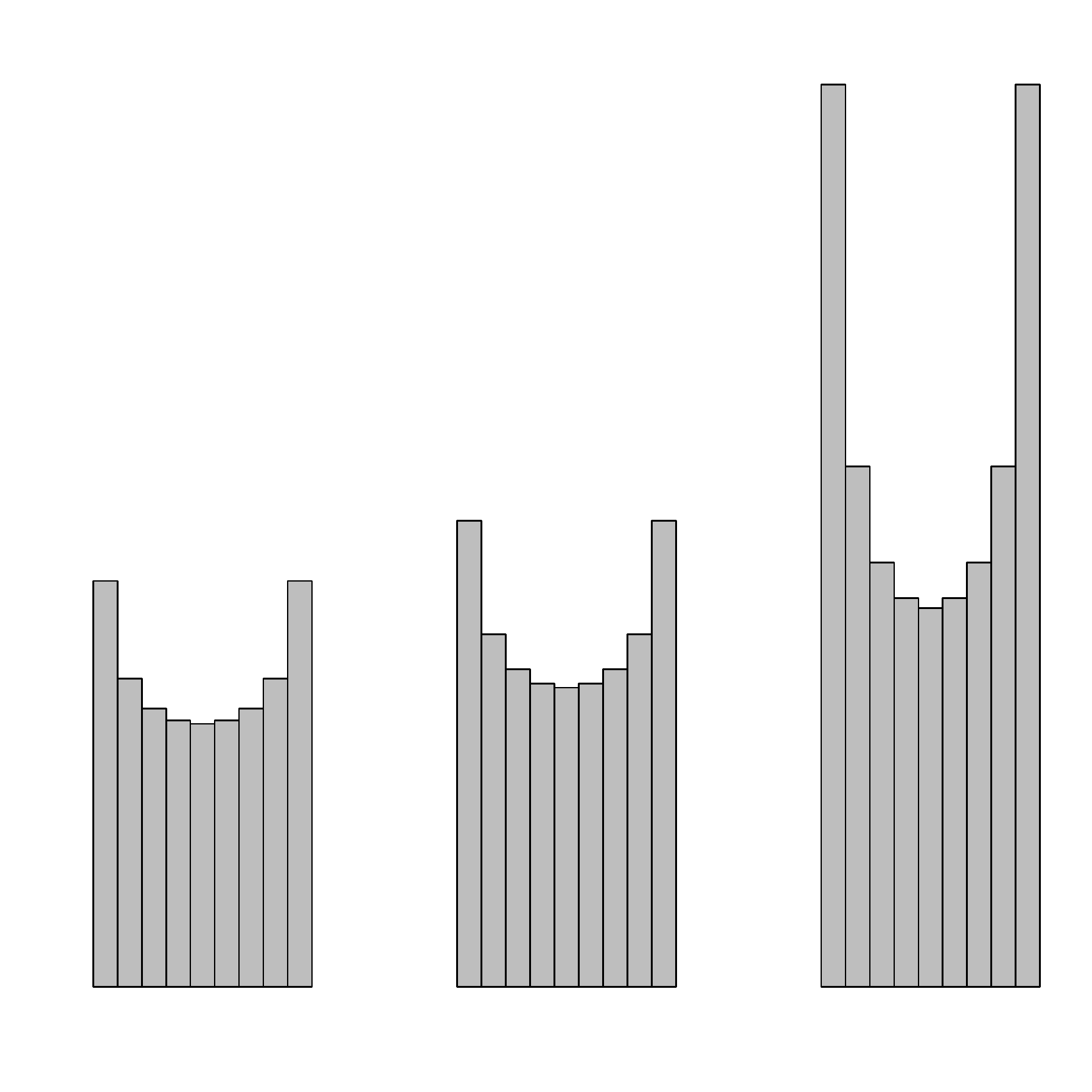} 
  \caption{\textit{Tightness ratio} using the same vertical axis	 for: (left) Binomial, (middle) Hypergeometric(4N,2N,N) and (right) Hypergeometric(2N,N,N).} 
  \label{fig:intuition2}
\end{figure}

\begin{definition}
  \label{def:2nd-order-tighter}
  A variable $\gamma \in \{0,1,\ldots,N\}$ is
  \begin{enumerate}
    \item \textit{2nd-order tighter than the Binomial}, if 
    $$\frac{P(\gamma=i)^{2}}{P(\gamma=i+1)P(\gamma=i-1)} > \frac{(i+1)(N-i+1)}{i(N-i)} \text{, for } 1 \leq i \leq N-1.$$
    \item \textit{2nd-order looser than the Binomial}, if
    $$\frac{P(\gamma=i)^{2}}{P(\gamma=i+1)P(\gamma=i-1)} < \frac{(i+1)(N-i+1)}{i(N-i)} \text{, for } 1 \leq i \leq N-1.$$
  \end{enumerate}
\end{definition}

For every $M$ and $N \leq a \leq M$, the Hypergeometic$(N+M,a,N)$ is 2nd-order tighter than the Binomial. The following Theorem shows that distributions that are 2nd-order tighter than the Binomial and the Hypergeometric relate in the same way to the \textit{belief in maturity}. 

\begin{theorem}[belief in maturity]
  \label{theorem:maturity}
  If $\gamma$ is 2nd-order tighter (looser) than the Binomial, then (reverse) belief in maturity holds.
\end{theorem}

Proposition \ref{proposition:extendibility-maturity} relates the assumption of 2nd-order looser than the Binomial to the assumption of infinite exchangeability in RW.

\begin{proposition}
  \label{proposition:extendibility-maturity}
  If $\texttt{X}_{N}$ can be extended to an infinitely exchangeable sequence and de Finetti's representation parameter is not degenerate, then $\gamma$ is 2nd-order looser than the Binomial. 
\end{proposition}

Hence, similarly to Theorem \ref{theorem:gambler} and Proposition \ref{proposition:extendibility-gambler} being an extension of the result in OP, Theorem \ref{theorem:maturity} and Propostion \ref{proposition:extendibility-maturity} are an extension of the result in RW.

The following proposition shows that the family CMP-Binomial$(N,p,\nu)$ \citep{shmueli2005} is sufficiently general to accomodate both distributions that are 2nd-order tighter than the Binomial and that are 2nd-order looser than the Binomial.

\begin{proposition}
  \label{proposition:CMP-Binomial-2}
  Let $\gamma \sim \text{CMP-Binomial}(N,p,\nu)$. If $\nu > 1$, then $\gamma$ is 2nd-order tighter than the Binomial. If $\nu < 1$, then $\gamma$ is 2nd-order looser than the Binomial.
\end{proposition}

Hence, it follows from Theorem \ref{theorem:maturity} and Proposition \ref{proposition:CMP-Binomial-2} that, according to the choice of $\nu$, the CMP-Binomial class can be used to model the \textit{belief in maturity} or the \textit{belief in reverse maturity}.

\section{Discussion}
\label{discussion}

The \textit{law of maturity} is a belief that is commonly combined with a belief in the symmetry between observations. For example, consider cases in which:

\begin{itemize}
  \item Equally fit individuals are competing for limited resources. For example, consider a sweepstake in which a company allocates a fixed budget to the prizes \citep{kalra2010} or a vase in which many seedlings have been planted \citep{diniz2010, kadane2014}. In these situations one might believe that any particular combination of the individuals is equally prone to succeed, but that the sparsity of resources creates a negative association between the successes of individuals. 
  \item A latent feature that controls the outcomes of an experiment varies accross observations. For example, consider the gender of pigs in a litter \citep{brooks1991}. It has been hypothesized that the sex of a pig zygote depends on the hormone levels of the parents. Fluctuations on the hormone levels generate negative association in the gender of the pigs. Indeed, \citet{brooks1991} finds that the number of male pigs per litter is underdispersed with respect to the Binomial distribution.
\end{itemize}

Despite the reasonability of the beliefs in maturity and in symmetry in the above situations, the assumption of infinite exchangeability (a common representation of symmetry) cannot be combined coherently with the \textit{law of maturity} (RW, OP). The contradiction between these two assumptions lead us investigate whether there exist probabilistic models that reconcile the beliefs in the \textit{law of maturity} and in the symmetry between observations.

As a specific contribution, Theorems \ref{theorem:gambler} and \ref{theorem:maturity} show that the \textit{law of maturity} is supported by the assumption of finite exchangeability. Under the assumption of finite exchangeability, if $\sum_{i=1}^{N}{X_{i}}$ is \textit{tighter than the Binomial}, then the \textit{law of maturity} holds. This condition is illustrated with the CMP-Binomial \citep{shmueli2005} parametric family in Propositions \ref{proposition:CMP-Binomial-1} and \ref{proposition:CMP-Binomial-2}. Hence, in situations in which one believes in the \textit{law of maturity} and in finite exchangeability, assuming that $\sum_{i=1}^{N}{X_{i}}$ follows the CMP-Binomial provides a more accurate probabilistic model than the ones obtained by the assumption of infinite exchangeability. 

As a more general contribution, we question the belief that there is no qualitative difference between assuming finite or infinite exchangeability. The combined results of RW, OP and Theorems \ref{theorem:gambler} and \ref{theorem:maturity} show that, in respect to the \textit{law of maturity}, such qualitative difference exists. While the \textit{law of maturity} is not supported by infinite exchangeability, it is supported by finite exchangeability. This example suggests caution in assuming infinite exchageability, specially when such an assumption depends on considering an hypothetical infinite population.

\bibliographystyle{abbrvnat}
\bibliography{maturity}

\section*{Proofs}

\begin{proof}[Proposition \ref{indifferent_belief_equiv_1}]
Consider that the coordinates of $\texttt{X}_{N}$ are jointly independent. Since $\texttt{X}_{N}$ if finitely exchangeable, the coordinates of $\texttt{X}_{N}$ are identically distributed. Thus, since $\gamma = \sum_{i=1}^{N}{X_{i}}$ and $X_{i}$ are i.i.d., conclude that $\gamma \sim  \text{Binomial}(N,P(X_{1}=1))$. Hence, under the assumption of independence, there exists $\pi \in [0,1]$ such that $\gamma \sim \text{Binomial}(N,\pi)$.

Also observe that, since $\texttt{X}_{N}$ is exchangeable, the distribution of $\texttt{X}_{N}$ is completely specified by the distribution of $\gamma$. Hence, there exists a unique distribution on $\texttt{X}_{N}$ for each distribution on $\gamma$. Conclude from the last paragraph that, if $\gamma \sim \text{Binomial}(N,\pi)$, then the coordinates of $\texttt{X}_{N}$ are independent. 

The proof of Proposition \ref{indifferent_belief_equiv_1} follows from the implications proved in the two previous paragraphs. 
\end{proof}

\begin{proof}[Propostion \ref{indifferent_belief_equiv_2}]

In order to prove Proposition \ref{indifferent_belief_equiv_2} we first show that the statement that $\texttt{X}_{N}$ models indifferent belief is equivalent to the joint independence of the coordinates of $\texttt{X}_{N}$. Observe that, by definition, the statement that $\texttt{X}_{N}$ models indifferent belief implies that the coordinates of $\texttt{X}_{N}$ are jointly independent. Also, since $\texttt{X}_{N}$ is finitely exchangeable, $P(X_{i} = 1) = P(X_{j} = 1)$. Hence, joint independence of the coordinates of $\texttt{X}_{N}$ imply that $\texttt{X}_{N}$ models indifferent belief.

The proof of Proposition \ref{indifferent_belief_equiv_2} follows from the equivalence that is proved in the previous paragraph and the direct application of Proposition \ref{indifferent_belief_equiv_1}. 
\end{proof}

\begin{lemma} Let $M = \min\{i \leq N: X_{i}=1\}$ be the first trial in which a success is observed. If $\gamma$ is tighter than the Binomial($N$, 1/2), then
  \label{lemma:predictive}
  \begin{align*}
    P(M=m|M \geq m) &> 1/2,	& \mbox{for $m=2,\ldots, N$}
  \end{align*}
  Similarly, if $\gamma$ is looser than the Binomial($N, 1/2$), then
  \begin{align*}
    P(M=m|M \geq m) &< 1/2,	& \mbox{for $m=2,\ldots, N$}
  \end{align*}
\end{lemma}

\begin{proof}[Lemma \ref{lemma:predictive}]
Let $t(k) = P(\gamma = k)/{N \choose k}$. Observe that
\begin{align*}
 P(M=m)={\displaystyle\sum_{\scriptscriptstyle k=1}^{\scriptscriptstyle N-m+1}}{N-m \choose k-1}t(k)
\end{align*}
Hence,
\begin{align*}
  P(M=m|M \geq m) &= \frac{\sum_{k=1}^{N-m+1}{{N-m \choose k-1}t(k)}}{t(0)+ \sum_{k=1}^{N-m+1}{\sum_{i=m}^{N-k+1}{{N-i \choose k-1} t(k)}}}	\\
  =\frac{{\sum_{k=1}^{N-m+1}}{N-m \choose k-1} t(k)}{t(0)+ {\sum_{k=1}^{N-m+1}} {N-m+1 \choose k} t(k) } &= \frac{{\sum_{k=1}^{N-m+1}}{N-m \choose k-1} t(k)}{{\sum_{k=1}^{N-m+1}} {N-m \choose k-1} (t(k)+t(\mbox{$k$$-$$1$}))} =
\end{align*}
$$= \frac{{\displaystyle\sum_{\scriptscriptstyle k=1}^{\scriptscriptstyle min(m-1,N-m+1)}}{N-m \choose k-1} t(k) + \left[{\displaystyle\sum_{\scriptscriptstyle k=m}^{\scriptscriptstyle N-m+1}}{N-m \choose k-1} t(k)\right] \mathbf{I}_{(m < \ N/2 + 1)} }{{\displaystyle\sum_{\scriptscriptstyle k=1}^{\scriptscriptstyle min(m-1,N-m+1)}} {N-m \choose k-1} (t(k)+t(\mbox{$k$$-$$1$})) + \left[{\displaystyle\sum_{\scriptscriptstyle k=m}^{\scriptscriptstyle N-m+1}} {N-m \choose k-1} (t(k)+t(\mbox{$k$$-$$1$}))\right] \mathbf{I}_{(m < \ N/2 + 1)}}$$\ \\

Consider the case in which $\gamma$ is tighter than the Binomial$(N,1/2)$. In order to prove the lemma it is sufficient to show the following: (1) the first sum in the numerator divided by the first sum in the denominator is greater than 1/2, (2) if $m < N/2+1$, the second sum in the numerator divided by the second sum in the denominator is greater than 1/2.

(1) If $k < (N+1)/2$, since $\gamma$ is tighter than the Binomial$(N,1/2)$, conclude that $\frac{t(k)}{t(k)+t(k-1)} > 1/2$. Hence, for every $1 \leq k \leq \min(m-1,N-m+1)$, $\frac{{N-m \choose k-1}t(k)}{{N-m \choose k-1}(t(k)+t(k-1))} > 1/2$.  

(2) Since $m < N/2+1$,
\begin{align}
  \label{eqn3}
  &\frac{{\displaystyle\sum_{\scriptscriptstyle k=m}^{\scriptscriptstyle N-m+1}}{N-m \choose k-1} t(k)}{{\displaystyle\sum_{\scriptscriptstyle k=m}^{\scriptscriptstyle N-m+1}} {N-m \choose k-1} (t(k)+t(\mbox{$k$$-$$1$})) } = \nonumber \\ 
  &= \frac{{\displaystyle\sum_{\scriptscriptstyle k=m}^{\scriptscriptstyle \lfloor N/2 \rfloor}}[{N-m \choose k-1} t(k)+{N-m \choose N-k}t(\mbox{$N$$-$$k$$+$$1$})]}{{\displaystyle\sum_{\scriptscriptstyle k=m}^{\scriptscriptstyle \lfloor N/2 \rfloor}} [{N-m \choose k-1} (t(k)+t(\mbox{$k$$-$$1$}))+{N-m \choose N-k}(t(\mbox{$N$$-$$k$$+$$1$})+t(\mbox{$N$$-$$k$}))]}
\end{align}

Notice that ${N-m \choose k-1} = {N-m \choose N-m-k+1}$. Also, since $m \leq N/2$ and $k < (N+1)/2$, ${N-m \choose N-k-m+1} > {N-m \choose N-k}$. Since $\gamma$ is symmetric, $\frac{t(k)+t(N-k+1)}{t(k)+t(k-1)+t(N-k+1)+t(N-k)} = 1/2$. Also, since $k < N/2+1$ and $\gamma$ is tighter than the Binomial$(N,1/2)$, conclude that $\frac{t(k)}{t(k)+t(k-1)} > 1/2$. Hence, in Equation \ref{eqn3} the ratio between each term in the numerator and each term in the denominator is greater than 1/2.

When $\gamma$ is looser than the Binomial$(N,1/2)$, $\frac{t(k)}{t(k)+t(k-1)} < 1/2$. Hence, all the inequalities are reversed. 
\end{proof}

\begin{lemma}
  \label{lemma:tight-implies-gambler}
  If $\gamma$ is tighter than the Binomial$(N,1/2)$, then the \textit{gambler's belief} holds. If $\gamma$ is looser than the Binomial$(N,1/2)$, then the \textit{reverse gambler's belief} holds.
\end{lemma}

\begin{proof}[Lemma \ref{lemma:tight-implies-gambler}]
  Without loss of generality, assume that the number of $0$'s in $\texttt{x}_{n}$ is larger than the number of $1$'s. Since the model is exchangeable, for any permutation $\pi$ of $\{1,\ldots,n\}$, $P(\gamma=\gamma_{0}|\texttt{X}_{i}=\texttt{x}_{i}) = P(\gamma=\gamma_{0}|\texttt{X}_{i}=\texttt{x}_{\pi(i)})$. Consider a permutation $\pi$ and $\texttt{y} = \texttt{x}_{\pi}$ such that, for some $a$, $\texttt{y}_{1}^{a}$ has an equal number of $0$'s and $1$'s and $\texttt{y}_{a+1}^{n}$ only has $0$'s. Let $\gamma^{*} = \sum_{i=a+1}^{N}{X_{i}}$.

\begin{align*}
  P(X_{n+1}=1|\texttt{x})	&= P(X_{n+1}=1|\texttt{y}) = \\
				&= \sum_{i=0}^{N}{P(X_{n+1}=1|\texttt{y}_{a+1}^{n},\gamma^{*}=i)P(\gamma^{*}=i|\texttt{y}_{1}^{a},\texttt{y}_{a+1}^{n})}
\end{align*}

That is, $P(X_{n+1}=1|\texttt{x})$ is equal to $P(X_{n+1}=1|\texttt{y}_{a+1}^{n})$ using $P(\gamma^{*}=i|\texttt{y}_{1}^{a})$ as a prior for $\gamma^{*}$. Observe that, 

\begin{align*}
  P(\gamma^{*}=i|y_{1}^{a})	&\propto P(\gamma^{*}=i,y_{1}^{a}) 					\\
				&=	 P(\gamma=i+a/2) {i+a/2 \choose a/2}{N-i-a/2 \choose a/2}
\end{align*}

 The last equality follows since $y_{1}^{a}$ has same number of $1$'s and $0$'s. Hence, if $\gamma$ is tighter (looser) than the Binomial$(N,1/2)$, then $\gamma^{*}|y_{1}^{a}$ is tighter (looser) than the Binomial$(N-a,1/2)$. Using $P(\gamma^{*}=i|\texttt{y}_{1}^{a})$ as a prior for $\gamma^{*}$, conclude from Lemma \ref{lemma:predictive} that, if $\gamma$ is tighter (looser) than the Binomial$(N,1/2)$, then $P(X_{n+1}=1|\texttt{y}_{a+1}^{n}) > (<) 1/2$.  
\end{proof}

\begin{lemma}
  \label{lemma:gambler-implies-tight}
  Assume the distribution of $\gamma$ is symmetric on $N/2$. If the (reverse) \textit{gambler's belief} holds, then $\gamma$ is tighter (looser) than the Binomial$(N,1/2)$.
\end{lemma}

\begin{proof}[Lemma \ref{lemma:gambler-implies-tight}]
  Let $\texttt{x}_{N-1} \in \{0,1\}^{N-1}$ and $s_{N-1}-1 = \sum_{i=1}^{N-1}{\texttt{x}_{N-1}}$.
  \begin{align*}
    &	P(X_{N}=1|\texttt{X}_{N-1}=\texttt{x}_{N-1}) =														\\
    &= \frac{P(X_{N}=1,\texttt{X}_{N-1}=\texttt{x}_{N-1})}{P(X_{N}=0,\texttt{X}_{N-1}=\texttt{x}_{N-1})+P(X_{N}=1,\texttt{X}_{N-1}=\texttt{x}_{N-1})}	\\
    &= \frac{P(\gamma=s_{N-1})/{N \choose s_{N-1}}}{P(\gamma=s_{N-1}-1)/{N \choose s_{N-1}-1} +P(\gamma=s_{N-1})/{N \choose s_{N-1}}}			\\
    &= \frac{1}{1 + \frac{P(\gamma=s_{N-1}-1)}{P(\gamma=s_{N-1})} \frac{N-s_{N-1}+1}{s_{N-1}}}
  \end{align*}
  If the gambler's relief holds, then $P(X_{N}=1|\texttt{X}_{N-1}=\texttt{x}_{N-1}) > \frac{1}{2}$, for every $s_{N-1} \leq \frac{N}{2}$. Hence, for every $s_{N-1} \leq \frac{N}{2}$, $\frac{P(\gamma=s_{N-1})}{P(\gamma=s_{N-1}-1)} > \frac{N-s_{N-1}+1}{s_{N-1}}$. Since the distribution of $\gamma$ is symmetric around $N/2$, conclude that $\gamma$ is tighter than the Binomial$(N,1/2)$. Similarly, if the reverse gambler's belief holds, then $\frac{P(\gamma=s_{N-1})}{P(\gamma=s_{N-1}-1)} < \frac{N-s_{N-1}+1}{s_{N-1}}$ and $\gamma$ is looser than the Binomial$(N,1/2)$. 
\end{proof}

\begin{proof}[Theorem \ref{theorem:gambler}]
  Follows from Lemmas \ref{lemma:tight-implies-gambler} and \ref{lemma:gambler-implies-tight}. 
\end{proof}

\begin{proof}[Proposition \ref{proposition:extendibility-gambler}]
  Since the distribution of de Finetti's parameter is exchangeable, the distribution of $\gamma$ is symmetric with respect to $N/2$. Hence, it remains to show that $P(\gamma=i)/P(\gamma=i-1) < (N-i+1)/y$, for $1 \leq i \leq N/2$. First, observe that, for every $0 \leq \pi \leq 1$ such that $\pi \neq 0.5$ and $n \geq 0$, it follows that $(\pi^{n+1}-(1-\pi)^{n+1})(\pi-(1-\pi)) > 0$. Hence, developing this expression, $\pi (1-\pi)^{n+1} + (1-\pi)\pi^{n+1} < \pi^{n+2} + (1-\pi)^{n+2}$. Hence, for $i \leq N/2$,
  
  \begin{align}
    \label{eqn1:point_inequality}
    \frac{\pi^{i}(1-\pi)^{N-i} + (1-\pi)^{i}\pi^{N-i}}{\pi^{i-1}(1-\pi)^{N-i+1} + (1-\pi)^{i-1}\pi^{N-i+1}} < 1
  \end{align}
  
  Next, since $\texttt{X}_{n}$ can be extended to an infinitely exchangeable sequence $\texttt{X}$, one can apply de Finetti's representation theorem \citep{finetti1931} to $\gamma$. That is, there exists a distribution $Q$ on $[0,1]$ such that, for every $i \in \mathbb{N}$, $P(\gamma = i) = \int_{0}^{1}{{N \choose i} \pi^{i}(1-\pi)^{N-i}Q(d\pi)}$. Thus,
  
  \begin{align*}
    \frac{P(\gamma=i)}{P(\gamma=i-1)}	&= \frac{\int_{0}^{1}{{N \choose i} \pi^{i}(1-\pi)^{N-i}Q(d\pi)}}{\int_{0}^{1}{{N \choose i-1} \pi^{i-1}(1-\pi)^{N-i+1}Q(d\pi)}}								\\
					&= \frac{\int_{0}^{0.5}{{N \choose i} (\pi^{i}(1-\pi)^{N-i}+((1-\pi)^{i}\pi^{N-i})Q(d\pi)}}{\int_{0}^{0.5}{{N \choose i-1} (\pi^{i-1}(1-\pi)^{N-i+1}+(1-\pi)^{i-1}\pi^{N-i+1}) Q(d\pi)}}	\\
					&< \frac{{N \choose i}}{{N \choose i-1}} = \frac{N-i+1}{i}																		
  \end{align*}
  The last inequality follows from Equation \ref{eqn1:point_inequality}. 
\end{proof}

\begin{proof}[Theorem \ref{theorem:maturity}]
  Let $M = \min\{i \leq N: X_{i}=1\}$ be the first trial in which a success is observed and $r(m) = P(M=m|M \geq m)$. In order to verify \textit{belief in maturity}, one must show that $r(m)$ increases on $m$. Let $t(k) = P(\gamma=k)/{N \choose k}$. Using the same development as in the proof of Lemma \ref{lemma:predictive},

  \begin{align}
    \label{eqn1:maturity}
    r(m)	&= \frac{{\displaystyle\sum_{\scriptscriptstyle k=1}^{\scriptscriptstyle N-m+1}}{N-m \choose k-1} t(k)}{{\displaystyle\sum_{\scriptscriptstyle k=0}^{\scriptscriptstyle N-m+1}} {N-m+1 \choose k} t(k)} \nonumber \\
		&= \frac{{\displaystyle\sum_{\scriptscriptstyle k=1}^{\scriptscriptstyle N-m}}{N-m-1 \choose k-1} t(k)+{\displaystyle\sum_{\scriptscriptstyle k=2}^{\scriptscriptstyle N-m+1}}{N-m-1 \choose k-2} t(k)}{{\displaystyle\sum_{\scriptscriptstyle k=0}^{\scriptscriptstyle N-m}} {N-m \choose k} t(k)+{\displaystyle\sum_{\scriptscriptstyle k=1}^{\scriptscriptstyle N-m+1}} {N-m \choose k-1} t(k)}
  \end{align}

  Observe that, in Equation \ref{eqn1:maturity}, the first sum in the numerator divided by the first sum in the denominator equals to $r(m+1)$. Therefore, to obtain $r(m+1) > r(m)$, it is sufficient to prove the following: 

  \begin{align*}
    \frac{{\displaystyle\sum_{\scriptscriptstyle k=1}^{\scriptscriptstyle N-m}}{N-m-1 \choose k-1} t(k)}{{\displaystyle\sum_{\scriptscriptstyle k=0}^{\scriptscriptstyle N-m}} {N-m \choose k} t(k)} > \frac{{\displaystyle\sum_{\scriptscriptstyle k=2}^{\scriptscriptstyle N-m+1}}{N-m-1 \choose k-2} t(k)}{{\displaystyle\sum_{\scriptscriptstyle k=1}^{\scriptscriptstyle N-m+1}} {N-m \choose k-1} t(k)}
  \end{align*}

which is equivalent to

  \begin{align}
    \label{eqn2:maturity}
    \frac{{\displaystyle\sum_{\scriptscriptstyle k=1}^{\scriptscriptstyle N-m}}{N-m-1 \choose k-1} t(k)}{{\displaystyle\sum_{\scriptscriptstyle k=1}^{\scriptscriptstyle N-m}} {N-m-1 \choose k-1} (t(k)+t(k-1))} > \frac{{\displaystyle\sum_{\scriptscriptstyle k=2}^{\scriptscriptstyle N-m+1}}{N-m-1 \choose k-2} t(k)}{{\displaystyle\sum_{\scriptscriptstyle k=2}^{\scriptscriptstyle N-m+1}} {N-m-1 \choose k-2} (t(k)+t(k-1))}
  \end{align}

If $\gamma$ is 2nd-order tighter than the Binomial, for every $1 \leq k \leq N-1$,

  \begin{align*}
    \frac{P(\gamma=k+1)/P(\gamma=k)}{P(\gamma=k)/P(\gamma=k-1)} < \frac{(N-y)/(y+1)}{(N-y+1)/y}
  \end{align*}

Hence, for every $1 \leq k \leq N-1$,

  \begin{align}
    \label{eqn3:maturity}
    \frac{t(k)}{t(k)+t(k-1)} > \frac{t(k+1)}{t(k+1)+t(k)}
  \end{align}

Hence, if $\gamma$ is 2nd-order tighter than the Binomial, Equation \ref{eqn3:maturity} holds and, therefore, Equation \ref{eqn2:maturity} also holds. Hence, if $\gamma$ is 2nd-order tighter than the Binomial, \textit{belief in maturity} holds. If $\gamma$ is 2nd-order looser than the Binomial, the proof follows by reversing the inequality in Equation \ref{eqn3:maturity}. 

\end{proof}

\begin{proof}[Proof of Proposition \ref{proposition:extendibility-maturity}]
  Since $\texttt{X}_{n}$ can be extended to an infinitely exchangeable sequence $\texttt{X}$, one can apply de Finetti's representation theorem \citep{finetti1931} to $\gamma$. That is, there exists a distribution $Q$ on $[0,1]$ such that, for every $i \in \mathbb{N}$, $P(\gamma = i) = \int_{0}^{1}{{N \choose i} \pi^{i}(1-\pi)^{N-i}Q(d\pi)}$. Hence,
  
  \begin{align*}
    &	\frac{P(\gamma=i)^{2}}{P(\gamma=i+1)P(\gamma=i-1)}																					\\
    &=	\frac{\left(\int_{0}^{1}{{N \choose i}\pi^{i}((1-\pi)^{N-i}Q(d\pi)}\right)^{2}}{\int_{0}^{1}{{N \choose i+1}\pi^{i+1}((1-\pi)^{N-i-1}Q(d\pi)}\int_{0}^{1}{{N \choose i-1}\pi^{i-1}((1-\pi)^{N-i+1}Q(d\pi)}}	\\
    &=	\frac{E_{Q}[\pi^{i}(1-\pi)^{N-i}]^{2}}{E_{Q}[\pi^{i+1}(1-\pi)^{N-i-1}]E_{Q}[\pi^{i-1}(1-\pi)^{N-i+1}]} \cdot \frac{(i+1)(N-i+1)}{i(N-i)}										\\
    &<	\frac{(i+1)(N-i+1)}{i(N-i)}
  \end{align*}
  The last line follows from the Cauchy-Schwarz inequality. 
  
\end{proof}

\end{document}